\documentclass[12pt]{article}
\usepackage[english]{babel}
\usepackage{latexsym,dsfont}
\usepackage{fullpage}
\usepackage[T1]{fontenc}
\usepackage{amsmath}
\usepackage{amsfonts}
\usepackage{amssymb}
\usepackage{mathrsfs}
\usepackage{amsthm}
\usepackage{latexsym}
\usepackage{array}
\usepackage{mathrsfs}
\usepackage{amsxtra}
\usepackage{amscd}
\usepackage{mathtools}
\usepackage{verbatim}
\usepackage{stmaryrd}
\usepackage{enumerate}
\usepackage{frcursive}
\usepackage{hyperref}
\hypersetup{
    colorlinks=true,
    linkcolor=blue,
	linktocpage=true
}

\numberwithin{equation}{section}

\theoremstyle{plain}
\newtheorem{Theorem}{Theorem}[section]
\newtheorem{Definition}[Theorem]{Definition}
\newtheorem{Proposition}[Theorem]{Proposition}
\newtheorem{Lemma}[Theorem]{Lemma}

\newtheorem{Remark}[Theorem]{Remark}

\newtheorem{innerAssumption}{Assumption}
\newenvironment{Assumption}[1]
{\renewcommand\theinnerAssumption{#1}\innerAssumption}
{\endinnerAssumption}

\def \E{\mathbb{E}}

\def \N{\mathbb{N}}
\def \P{\mathbb{P}}
\def \R{\mathbb{R}}

\def \Fc{{\cal F}}

\def \Nc{{\cal N}}

\def \Wc{{\cal W}}

\begin{document}

\title{On smooth approximations in the Wasserstein space}

\author{
Andrea COSSO\footnote{Universit\`a degli Studi di Milano, Dipartimento di Matematica, via Saldini 50, 20133 Milano, Italy; \texttt{andrea.cosso@unimi.it}} \quad Mattia MARTINI\footnote{Universit\'e C\^ote d'Azur, CNRS, Laboratoire J.A. Dieudonn\'e, 06108 Nice, France; \texttt{mattia.martini@univ-cotedazur.fr}} \footnote{M. Martini conducted this research during his PhD at Universit\`a degli Studi di Milano. Currently, M. Martini is supported by the European Research Council (ERC) under the European Union Horizon 2020 research and innovation program (ELISA project, Grant agreement nr. 101054746).}}

\maketitle

\begin{abstract}
\noindent In this paper we investigate the approximation of continuous functions on the Wasserstein space by smooth functions, with smoothness meant in the sense of Lions differentiability. In particular, in the case of a Lipschitz function we are able to construct a sequence of infinitely differentiable functions having the same Lipschitz constant as the original function. This solves an open problem raised in \cite{MZ}. For (resp. twice) continuously differentiable function, we show that our approximation also holds for the first-order derivative (resp. second-order derivatives), therefore solving another open problem raised in \cite{MZ}.
\end{abstract}

\vspace{5mm}

\noindent {\bf Keywords:} Wasserstein space, Lions differentiability, smooth approximations.

\vspace{5mm}

\noindent {\bf Mathematics Subject Classification (2020):} 28A33, 28A15, 49N80

\section{Introduction}

In the present paper our aim is to find smooth approximations of continuous functions $u\colon\mathscr P_2(\R^d)\rightarrow\R$, where $\mathscr P_2(\R^d)$ is the space of probability measures on $(\R^d,\mathscr B(\R^d))$ endowed with the $2$-Wasserstein metric $\Wc_2$ (see \eqref{W2}). Here smooth is meant in the sense of Lions differentiability, whose main features are recalled in Section \ref{S:LionsDiff}. We show that when $u$ is continuous there exists a sequence $\{u_k\}_k\in C^\infty(\mathscr P_2(\R^d))$ which converges to $u$ uniformly on compact subsets of $\mathscr P_2(\R^d)$. If in addition $u$ is uniformly (resp. Lipschitz) continuous then each $u_k$ is also uniformly (resp. Lipschitz) continuous, with the same modulus of continuity (resp. Lipschitz constant) as $u$, therefore solving an open problem raised in \cite[Remark 3.2-(i)]{MZ}. Moreover, for $u\in C^1(\mathscr P_2(\R^d))$ (resp. $u\in C^2(\mathscr P_2(\R^d))$) we show that the convergence also holds for the first-order derivative (resp. second-order derivatives), therefore solving another open problem raised in \cite[Remark 3.2-(ii)]{MZ}.

The smooth approximating sequence $\{u_k\}_k\in C^\infty(\mathscr P_2(\R^d))$ is constructed relying on the empirical distribution, similarly to what is done in \cite{CD18_I,Martini1} in the proof of It\^o's formula along a flow of measures (see \cite[Theorem 5.92]{CD18_I} and \cite[Proposition 5.1]{Martini1}), although there it is not really a smoothing as both $u$ and $u_k$ are of class $C^2(\mathscr P_2(\R^d))$, but rather a way to approximate a function $u$ on $\mathscr P_2(\R^d)$ by functions defined on finite-dimensional spaces. On the other hand, in \cite[Section 3]{MZ} an approximating sequence $\{u_k\}_k\in C^\infty(\mathscr P_2(\R^d))$ for a continuous function $u$ is built relying on the idea of discretizing a generic probability measure $\mu\in\mathscr P_2(\R^d)$ by a sequence of discrete distributions $\{\mu_k\}_k\subset\mathscr P_2(\R^d)$. However, as mentioned in \cite[Remark 3.2]{MZ}, such an approximation does not satisfy certain properties (unless imposing stronger assumptions on $u$) that are instead satisfied by our approximating sequence.

The existence of a smooth approximating sequence for continuous functions $u\colon\mathscr P_2(\R^d)\rightarrow\R$ plays a relevant role in different contexts, as for instance in the study of mean field games or mean field control problems (for which we refer to \cite{carda12,CD18_I,Lions}). We mention in particular \cite{TTZ} where the approximating sequence constructed in \cite{MZ} was used and forced the authors to impose stronger assumptions on the coefficients of the mean field optimal stopping problem, which can be avoided by relying on our approximating sequence. We also mention \cite{CKGPR_Uniq} and \cite{Martini2}, where a smooth approximating sequence is used to prove the comparison theorem for viscosity solutions of the Hamilton-Jacobi-Bellman of the mean field control problem and of the backward Kolmogorov equation associated to the nonlinear filtering equation. We finally mention again \cite{CD18_I,Martini1} where a smooth approximating is constructed in order to prove It\^o's formula along a flow of probability measures.

The rest of the paper is organized as follows. In Section \ref{S:Setting} we introduce the notations, the probabilistic setting, and the Wasserstein space $(\mathscr P_2(\R^d),\Wc_2)$. We also prove some specific properties of such a space that are used in the subsequent sections. Section \ref{S:LionsDiff} is devoted to recall the main features of Lions differentiability. Finally, Section \ref{S:SmoothApprox} contains the main results of the paper, namely Theorem \ref{T:Approx}, concerning the smooth approximations of a continuous function $u\colon\mathscr P_2(\R^d)\rightarrow\R$, and Theorem \ref{T:Approx_Deriv}, regarding the approximation of first and second-order derivatives.

\section{Probabilistic setting and Wasserstein space}
\label{S:Setting}

Let $(\Omega,\Fc,\P)$ be a probability space. Given a random variable $\xi\colon\Omega\rightarrow\R^d$ we denote by $\P_\xi$ its law on $(\R^d,\mathscr B(\R^d))$, where $\mathscr B(\R^d)$ denotes the Borel $\sigma$-algebra on $\R^d$. Moreover, we denote by $L^2(\Omega,\Fc,\P;\R^d)$ (or simply $L^2$) the space of (equivalence classes of) $\R^d$-valued random variables on $(\Omega,\Fc,\P)$. We also denote by $\|\xi\|_{L^2}$ the $L^2$-norm of a random variable $\xi\in L^2(\Omega,\Fc,\P;\R^d)$.\\
In the sequel we always impose the following assumption on the probability space $(\Omega,\Fc,\P)$.

\begin{Assumption}{\bf(H$_U$)}\label{A:H_U}
There exists a real-valued random variable $U$ defined on $(\Omega,\Fc,\P)$ having uniform distribution on $[0,1]$.
\end{Assumption}

\noindent The above assumption is equivalent to other requirements on $(\Omega,\Fc,\P)$, as established in the following lemma.

\begin{Lemma}\label{L:H_U}
Let $(\Omega,\Fc,\P)$ be a probability space. Then, for every $d\in\N$ and $p\in[1,\infty)$, the following statements are equivalent.
\begin{enumerate}[\upshape1)]
\item $(\Omega,\Fc,\P)$ satisfies Assumption \ref{A:H_U}.
\item For every probability $\mu$ on $(\R^d,\mathscr B(\R^d))$, with finite $p$-th moment, that is $\int_{\R^d}|x|^p\mu(dx)$ $<$ $\infty$, there exists an $\R^d$-valued random variable $\xi$ on $(\Omega,\Fc,\P)$ with distribution $\mu$.
\item For every probability $\mu$ on $(\R^d,\mathscr B(\R^d))$ there exists an $\R^d$-valued random variable $\xi$ on $(\Omega,\Fc,\P)$ with distribution $\mu$, that is $\P_\xi=\mu$.
\item $(\Omega,\Fc,\P)$ is atomless: for every $A\in\Fc$, with $\P(A)>0$, there exists $B\in\Fc$, with $B\subset A$, such that $0<\P(B)<\P(A)$.
\end{enumerate}
\end{Lemma}
\begin{proof}
The equivalence between items 1) and 2), and similarly between items 1) and 3), can be proved proceeding along the same lines as in the proof of \cite[Lemma 2.1]{CKGPR} (notice that in \cite[Lemma 2.1]{CKGPR} the probability $\mu$ has finite second moment; however, the same proof works for a generic probability $\mu$).\\
Let us prove the equivalence between items 1) and 4). The implication 4) $\Longrightarrow$ 1) follows from \cite[Proposition 9.1.11]{Bogachev}. On the other hand, let $A\in\Fc$ be such that $\P(A)>0$ and define $F\colon[0,1]\rightarrow[0,1]$ as $F(x)$ $=$ $\P(A\cap\{U\leq x\})$, $\forall\,x\in[0,1]$. Notice that $F(0)=0$ and $F(1)=\P(A)$. Moreover, if $x_n\downarrow x$ then $\lim_n F(x_n)=\lim_n\P(A\cap\{U\leq x_n\})=\P(A\cap\{U\leq x\})=F(x)$; similarly, if $x_n\uparrow x$ then $\lim_n F(x_n)=\lim_n\P(A\cap\{U\leq x_n\})=\P(A\cap\{U<x\})=\P(A\cap\{U\leq x\})=F(x)$. This shows that $F$ is a continuous function. Then, there exists $b\in(0,1)$ such that $F(b)>0$, that is $\P(B)>0$, where $B=A\cap\{U\leq b\}$.
\end{proof}

\noindent For every $d\in\N$, let $\mathscr P(\R^d)$ denote the family of all probabilities on $(\R^d,\mathscr B(\R^d))$.	We define
\[
\mathscr P_2(\R^d) \ := \ \left\{\mu \in \mathscr P(\R^d)\colon \int_{\R^d} |x|^2\,\mu(dx)<+\infty\right\}.
\]

\begin{Remark}
Since $(\Omega,\Fc,\P)$ satisfies Assumption \ref{A:H_U}, by Lemma \ref{L:H_U} it holds that
\begin{equation}\label{H_U}
\mathscr P_2(\R^d) \ = \ \big\{\P_\xi\colon\xi\in L^2(\Omega,\Fc,\P;\R^d)\big\}.
\end{equation}
\end{Remark}

\noindent Given $\mu,\nu\in\mathscr P_2(\R^d)$, a probability $\pi$ on $(\R^d\times\R^d,\mathscr B(\R^d\times\R^d))$ is called \emph{coupling} (or \emph{transport plan}) of $\mu$ and $\nu$ if $\pi$ has $\mu$ as first marginal and $\nu$ as second marginal, namely $\mu(B)$ $=$ $\pi(B\times\R^d)$ and $\nu(B)$ $=$ $\pi(\R^d\times B)$, for every $B\in\mathscr B(\R^d)$. We denote by $\Pi(\mu,\nu)$ the family of all couplings of $\mu$ and $\nu$. Notice that $\Pi(\mu,\nu)$ is non-empty since the product measure $\mu\otimes\nu$ belongs to $\Pi(\mu,\nu)$. We endow $\mathscr P_2(\R^d)$ with the $2$-Wasserstein metric
\begin{equation}\label{W2}
\Wc_2(\mu,\nu) \ := \ \inf_{\pi\in\Pi(\mu,\nu)}\bigg\{\bigg(\int_{\R^d\times\R^d}|x-y|^2\,\pi(dx,dy)\bigg)^{1/2}\bigg\}.
\end{equation}
It can be shown that $\Wc_2$ is a metric on $\mathscr P_2(\R^d)$ and $(\mathscr P_2(\R^d),\Wc_2)$ is a complete separable metric space, see \cite[Chapter 7]{AGS08} or \cite[Chapter 6]{Vi09}.

\begin{Remark}
Since $(\Omega,\Fc,\P)$ satisfies Assumption \ref{A:H_U}, by Lemma \ref{L:H_U} it holds that
\begin{equation}\label{W_2_xi_eta}
\Wc_2(\mu,\nu) \ = \ \inf_{\substack{\xi,\eta\in L^2\\\P_\xi=\mu,\,\P_\eta=\nu}}\|\xi-\eta\|_{L^2}.
\end{equation}
Then, given $\xi,\eta\in L^2(\Omega,\Fc,\P;\R^d)$ with laws $\mu$ and $\nu$, respectively, it follows directly from \eqref{W_2_xi_eta} that
\begin{equation}\label{EstimateW2}
\Wc_2(\mu,\nu) \ \leq \ \|\xi-\eta\|_{L^2}.
\end{equation}
\end{Remark}

\noindent We denote $\|\mu\|_2$ $:=$ $\Wc_2(\mu,\delta_0)$ $=$ $\big(\int_{\R^d}|x|^2\,\mu(dx)\big)^{1/2}$, for every $\mu\in\mathscr P_2(\R^d)$, where $\delta_0$ is the Dirac delta centered at $0$. We observe that if $\xi\in L^2(\Omega,\Fc,\P;\R^d)$ has distribution $\mu$, then
\begin{equation}\label{NormWass}
\|\mu\|_2 \ = \ \|\xi\|_{L^2}.
\end{equation}
We end this section with a useful technical result.
\begin{Lemma}
Let $n\in\N$ and $x_1,\ldots,x_n,y_1,\ldots,y_n\in\R^d$. Then
\begin{equation}\label{EstimateW2_bis}
\Wc_2\bigg(\frac{1}{n}\sum_{i=1}^n\delta_{x_i},\frac{1}{n}\sum_{i=1}^n\delta_{y_i}\bigg) \ \leq \ \sqrt{\frac{1}{n}\sum_{i=1}^n|x_i - y_i|^2}.
\end{equation}
\end{Lemma}
\begin{proof}
	Let $\xi$ be a random variable with law $\frac{1}{n}\sum_{i=1}^n\delta_{x_i}$. Let us introduce the mapping $\psi\colon\{x_1,\ldots,x_n\}\to\{y_1,\ldots,y_n\}$ such that $\psi(x_i) = y_i$ for $i=1,\ldots n$. Then, the random variable $\eta:=\psi(\xi)$ has law $\frac{1}{n}\sum_{i=1}^n\delta_{y_i}$. Thus, from \eqref{EstimateW2} it follows that
	\begin{equation*}
		\Wc_2\bigg(\frac{1}{n}\sum_{i=1}^n\delta_{x_i},\frac{1}{n}\sum_{i=1}^n\delta_{y_i}\bigg) \ \leq \ \sqrt{\E\left[\lvert \xi - \eta\rvert^2\right]} = \ \sqrt{\E\left[\lvert \xi - \psi(\xi)\rvert^2\right]} = \ \sqrt{\frac{1}{n}\sum_{i=1}^n|x_i - y_i|^2}.
	\end{equation*}
\end{proof}

\section{Lions differentiability}
\label{S:LionsDiff}

There are different definitions of differentiability for functions defined on the Wasserstein space $\mathscr P_2(\R^d)$. In the present paper we consider the notion introduced by P.-L. Lions in the series of lectures \cite{Lions}. In this section we recall the main features of such a notion, while we refer to \cite[Chapter 5]{CD18_I} for more details.

\begin{Definition}
Let $u\colon\mathscr P_2(\R^d)\rightarrow\R$. The function $\tilde u\colon L^2(\Omega,\Fc,\P;\R^d)\rightarrow\R$ given by
\[
\tilde u(\xi) \ = \ u(\P_\xi), \qquad \forall\,\xi\in L^2(\Omega,\Fc,\P;\R^d),
\]
is said to be the \textbf{lifting} of $u$.
\end{Definition}
	
\begin{Definition}
Let $u\colon\mathscr P_2(\R^d)\rightarrow\R$ and $\mu_0\in\mathscr P_2(\R^d)$. The function $u$ is said to be \textbf{Lions differentiable} or $L$-\textbf{differentiable} at $\mu_0$ if there exists $\xi_0\in L^2(\Omega,\Fc,\P;\R^d)$, with $\P_{\xi_0}=\mu_0$, and $\tilde u$ is Fr\'echet differentiable at $\xi_0$.
\end{Definition}
	
\begin{Definition}
Let $u\colon\mathscr P_2(\R^d)\rightarrow\R$. Suppose that its lifting $\tilde u$ is everywhere Fr\'echet differentiable. The function $u$ is said to admit $L$-\textbf{derivative} if there exists a function $\partial_\mu u$ on $\mathscr P_2(\R^d)$, with $\mu_0\mapsto\partial_\mu u(\mu_0)(\cdot)\in L^2(\R^d,\mathscr B(\R^d),\mu_0;\R^d)$ and
\[
D\tilde u(\xi_0) \ = \ \partial_\mu u(\mu_0)(\xi_0), \qquad \P\text{-a.s.}
\]
for every $\xi_0\in L^2(\Omega,\Fc,\P;\R^d)$, $\P_{\xi_0}=\mu_0$. We refer to $\partial_\mu u$ as the $L$-\textbf{derivative} of $u$.
\end{Definition}
	
\begin{Proposition}
Let $u\colon\mathscr P_2(\R^d)\rightarrow\R$. Suppose that its lifting $\tilde u$ is everywhere Fr\'echet differentiable and $D\tilde u\colon L^2(\Omega,\Fc,\P;\R^d)\rightarrow L^2(\Omega,\Fc,\P;\R^d)$ is a continuous function. Then, for every $\mu_0\in\mathscr P_2(\R^d)$, there exists a Borel-measurable function $g_{\mu_0}\colon\R^d\rightarrow\R^d$ such that
\[
D\tilde u(\xi_0) \ = \ g_{\mu_0}(\xi_0), \qquad \P\text{-a.s.}
\]
for every $\xi_0\in L^2(\Omega,\Fc,\P;\R^d)$ with $\P_{\xi_0}=\mu_0$.
\end{Proposition}
\begin{proof}
See \cite[Proposition 5.25]{CD18_I}.
\end{proof}

\begin{Proposition}
Let $u\colon\mathscr P_2(\R^d)\rightarrow\R$. Suppose that its lifting $\tilde u$ is everywhere Fr\'echet differentiable and $D\tilde u\colon L^2(\Omega,\Fc,\P;\R^d)\rightarrow L^2(\Omega,\Fc,\P;\R^d)$ is a continuous function. Then, there exists at most one map $\partial_\mu u\colon\mathscr P_2(\R^d)\rightarrow\R$ satisfying:
\begin{enumerate}[\upshape1)]
\item $\partial_\mu u$ is continuous on $\mathscr P_2(\R^d)\times\R^d$.
\item $\forall\,\xi_0\in L^2(\Omega,\Fc,\P;\R^d)$, $\P_{\xi_0}=\mu_0$,
\[
D\tilde u(\xi_0) \ = \ \partial_\mu u(\mu_0)(\xi_0), \qquad \P\text{-a.s.}
\]
\end{enumerate}
If such a function exists then we say that $u$ admits \textbf{continuous} $L$-\textbf{derivative}.
\end{Proposition}
\begin{proof}
See \cite[Remark 5.82]{CD18_I}.
\end{proof}

\begin{Definition}
Let $C^1(\mathscr P_2(\R^d))$ be the class of continuous functions $u\colon\mathscr P_2(\R^d)\rightarrow\R$ which admit continuous $L$-derivative.
\end{Definition}
	
\noindent Now, let $u\colon\mathscr P_2(\R^d)\rightarrow\R$ be of class $C^1(\mathscr P_2(\R^d))$. Then, we can consider the derivatives of the function $(x,\mu)\mapsto u(\mu)(x)$, that is the second-order derivatives of $u$:
\begin{itemize}
\item for every fixed $\mu\in\mathscr P_2(\R^d)$, the derivative of the function $x\mapsto u(\mu)(x)$, denoted by $\partial_x\partial_\mu u(\mu)(x)$:
\[
\partial_x\partial_\mu u\colon\mathscr P_2(\R^d)\times\R^d \ \longrightarrow \ \R^{d\times d};
\]
\item for every fixed $x\in\R^d$, the derivative of the function $\mu\mapsto u(\mu)(x)$, denoted by $\partial_\mu^2 u(\mu)(x,y)$:
\[
\partial_\mu^2 u\colon\mathscr P_2(\R^d)\times\R^d\times\R^d \ \longrightarrow \ \R^{d\times d}.
\] 
\end{itemize}

\begin{Definition}
Let $C^2(\mathscr P_2(\R^d))$ be the class of functions $u\in C^1(\mathscr P_2(\R^d))$ which admit continuous second-order derivatives $\partial_x\partial_\mu u$ and $\partial_\mu^2 u$.
\end{Definition}

\noindent Similarly, considering higher-order derivatives we can give the definition of the class $C^k(\mathscr P_2(\R^d))$, for every $k\in\N$. As an example, we report all the third-order derivatives:
\begin{align*}
\partial^2_x\partial_\mu u\colon\hspace{1.74cm}\mathscr P_2(\R^d)\times\R^d \ &\longrightarrow \ \R^{d\times d\times d},\\
\partial_\mu\partial_x\partial_\mu u\colon\hspace{8.7mm}\mathscr P_2(\R^d)\times\R^d\times\R^d \ &\longrightarrow \ \R^{d\times d\times d},\\
\partial_{x} \partial_\mu^2 u\colon\hspace{8.7mm}\mathscr P_2(\R^d)\times\R^d\times\R^d \ &\longrightarrow \ \R^{d\times d\times d},\\
\partial_{y}\partial_\mu^2 u\colon\hspace{8.7mm}\mathscr P_2(\R^d)\times\R^d\times\R^d \ &\longrightarrow \ \R^{d\times d\times d},\\
\partial_\mu^3 u\colon\mathscr P_2(\R^d)\times\R^d\times\R^d\times\R^d \ &\longrightarrow \ \R^{d\times d\times d}.
\end{align*}
For a careful analysis of higher-order Lions derivatives beyond order 2 we refer to the recent paper \cite{salkeld}, where a clever multi-index notation is introduced (see \cite[Section 2.1]{salkeld}) and also a Schwarz theorem for multivariate Lions derivatives is provided, see \cite[Theorem 3.9]{salkeld} (for the second-order case see \cite[Corollary 5.89]{CD18_I} and \cite[Remark 4.16]{CD18_II}). Finally, we give the following definition.

\begin{Definition}
Let $C^\infty(\mathscr P_2(\R^d))$ be the class of continuous functions $u\colon\mathscr P_2(\R^d)\rightarrow\R$ which admit continuous derivatives of any order.
\end{Definition}

\section{Smooth approximations in the Wasserstein space}
\label{S:SmoothApprox}

We begin recalling the following classical definition.

\begin{Definition}
We say that $w\colon[0,\infty)\rightarrow[0,\infty)$ is a modulus of continuity if $w$ is continuous, non-decreasing, concave, and $w(0)=0$.
\end{Definition}

\noindent We can now state one of the main results of the paper.

\begin{Theorem}\label{T:Approx}
Let $u\colon\mathscr P_2(\R^d)\rightarrow\R$ be a continuous function. Then, there exists a sequence $\{u_k\}_k$ in $C^\infty(\mathscr P_2(\R^d))$ such that
\begin{equation}\label{Conv_u_n_u}
\lim_{k\rightarrow\infty} \sup_{\mu\in\mathscr M} |u_k(\mu) - u(\mu)| \ = \ 0,
\end{equation}
for any compact set $\mathscr M\subset\mathscr P_2(\R^d)$. Moreover, if $u$ is uniformly continuous with modulus of continuity $w$ then every function $u_k$ is also uniformly continuous with the same modulus of continuity $w$.
\end{Theorem}
\begin{proof}
We split the proof into four steps.

\vspace{2mm}

\noindent\emph{Step I. Definition of $u\star\gamma_k$.} We proceed as in \cite[Lemma 5.95]{CD18_I} and consider a sequence $\{\gamma_k\}_k$ of $C^\infty$ functions $\gamma_k\colon\R^d\rightarrow\R^d$ having compact support and such that $(\gamma_k,D_x\gamma_k,D_x^2\gamma_k)(x)\rightarrow(x,I_d,0)$ uniformly on compact sets of $\R^d$ as $k\rightarrow\infty$, where $I_d$ stands for the identity matrix of order $d$ and $0$ denotes the zero matrix of dimension $d\times d$. We also suppose that $|D_x\gamma_k(x)|\leq1$, for every $k\in\N$ and $x\in\R^d$. Then, we define
\[
(u\star\gamma_k)(\mu) \ = \ u(\mu\circ\gamma_k^{-1}), \qquad \forall\,\mu\in\mathscr P_2(\R^d).
\]
Notice that $u\star\gamma_k$ is continuous and bounded (see \cite[Lemma 5.94]{CD18_I}). Moreover, if $u$ is uniformly continuous with modulus $w$ then $u\star\gamma_k$ is also uniformly continuous with the same modulus of continuity. As a matter of fact, it holds that
\[
\big|u(\mu\circ\gamma_k^{-1}) - u(\nu\circ\gamma_k^{-1})\big| \ \leq \ w\big(\Wc_2(\mu\circ\gamma_k^{-1},\nu\circ\gamma_k^{-1})\big).
\]
Recalling \eqref{W_2_xi_eta} we find
\[
\Wc_2(\mu\circ\gamma_k^{-1},\nu\circ\gamma_k^{-1}) \ = \ \inf_{\substack{\xi,\eta\in L^2\\\P_\xi=\mu,\,\P_\eta=\nu}}\|\gamma_k(\xi)-\gamma_k(\eta)\|_{L^2}.
\]
Since $|D_x\gamma_k(x)|\leq1$, for every $k\in\N$ and $x\in\R^d$, we get $\|\gamma_k(\xi)-\gamma_k(\eta)\|_{L^2}\leq\|\xi-\eta\|_{L^2}$, from which we conclude that
\begin{equation}\label{UC_u_rho_k}
\big|(u\star\gamma_k)(\mu) - (u\star\gamma_k)(\nu)\big| \ = \ \big|u(\mu\circ\gamma_k^{-1}) - u(\nu\circ\gamma_k^{-1})\big| \ \leq \ w\big(\Wc_2(\mu,\nu)\big).
\end{equation}

\vspace{2mm}

\noindent\emph{Step II. Definitions of $v_{k,n},v_{k,n,m}$ and pointwise convergence \eqref{Conv_n_m}.} For every $k,n\in\N$, let $v_{k,n}\colon\mathscr P_2(\R^d)\rightarrow\R$ be given by
\begin{equation}\label{v_n}
v_{k,n}(\mu) \ = \ \E\bigg[(u\star\gamma_k)\bigg(\frac{1}{n}\sum_{i=1}^n\delta_{\xi_i}\bigg)\bigg], \qquad \forall\,\mu\in\mathscr P_2(\R^d),
\end{equation}
where $\xi_1,\ldots,\xi_n$ are independent and identically distributed $\R^d$-valued random variables on $(\Omega,\Fc,\P)$ with distribution $\mu$. Let also $\rho\colon\R^d\rightarrow\R$ be the probability density function of the standard multivariate normal distribution, that is $\rho(x)$ $=$ $\frac{1}{(2\pi)^{d/2}}\,\textup{e}^{-\frac{1}{2}|x|^2}$, $\forall\,x\in\R^d$. For every $m\in\N$, let $\rho_m(x)=m^d\rho(mx)$, $\forall\,x\in\R^d$, and, given $n\in\N$, define $h_{k,n,m}\colon\R^{dn}\rightarrow\R$ as
\[
h_{k,n,m}(x_1,\ldots,x_n) \ = \ \int_{\R^{dn}}(u\star\gamma_k)\bigg(\frac{1}{n}\sum_{i=1}^n\delta_{x_i-y_i}\bigg)\rho_m(y_1)\cdots\rho_m(y_n)\,dy_1\cdots dy_n,
\]
for every $x_1,\ldots,x_n\in\R^d$. Notice that $h_{k,n,m}$ is of class $C^\infty(\R^{dn})$. Now, define $v_{k,n,m}\colon\mathscr P_2(\R^d)$ $\rightarrow\R$ as follows
\[
v_{k,n,m}(\mu) = \E\big[h_{k,n,m}(\xi_1,\ldots,\xi_n)\big] = \!\int_{\R^{dn}}\!\!\!\!\E\bigg[(u\star\gamma_k)\bigg(\frac{1}{n}\sum_{i=1}^n\delta_{\xi_i-y_i}\bigg)\bigg]\rho_m(y_1)\cdots\rho_m(y_n)dy_1\cdots dy_n,
\]
for every $\mu\in\mathscr P_2(\R^d)$, where $\xi_1,\ldots,\xi_n$ are independent and identically distributed $\R^d$-valued random variables on $(\Omega,\Fc,\P)$ with distribution $\mu$. Notice that $v_{k,n,m}$ is also given by
\[
v_{k,n,m}(\mu) \ = \ \E\bigg[(u\star\gamma_k)\bigg(\frac{1}{n}\sum_{i=1}^n\delta_{\xi_i-Z_i}\bigg)\bigg], \qquad \forall\,\mu\in\mathscr P_2(\R^d),
\]
where $Z_1,\ldots,Z_n$ are independent and identically distributed $\R^d$-valued random variables on $(\Omega,\Fc,\P)$ having multivariate normal distribution $\Nc(0,m^2 I_d)$ ($I_d$ denotes the identity matrix of order $d$) and being independent of $\xi_1,\ldots,\xi_n$. It is easy to see that $v_{k,n,m}\in C^\infty(\mathscr P_2(\R^d))$.  As a matter of fact, from \cite[formula (5.37)]{CD18_I} it follows that
\begin{align*}
\partial_\mu v_{k,n,m} (\mu)(x_1)  \ &= \ n\E\big[ \partial_{x_1}h_{k,n,m}(x_1,\xi_2,\dots,\xi_n)\big],\\
\partial^2_\mu v_{k,n,m} (\mu)(x_1,x_2) \ &= \ n (n-1)\E\big[ \partial_{x_2}\partial_{x_1}h_{k,n,m}(x_1,x_2,\xi_3,\dots,\xi_n)\big],\\
&\,\,\,\vdots\\
\partial_\mu^n v_{k,n,m} (\mu)(x_1,\dots,x_n) \ &= \ n!\,\partial_{x_n}\cdots\partial_{x_1}h_{k,n,m}(x_1,\dots,x_n),\\
\partial_\mu^{n+1} v_{k,n,m} (\mu)(x_1,\dots,x_n, x_{n+1}) \ &= \ 0.
\end{align*}
Moreover, from the formulae above we deduce that derivatives with respect to the spatial variables $x_1,\ldots,x_n,\ldots$ are also smooth, thanks to the smoothness of $h_{k,n,m}$. Furthermore, we observe that any higher-order Lions derivative involving more than $n$ derivatives in measure is identically equal to zero. Finally, since $h_{k,n,m}$ has compact support, all Lions derivatives of $v_{k,n,m}$ are bounded, although not necessarily uniformly with respect to $k,n,m$.\\

Now, notice that, for every $\mu\in\mathscr P_2(\R^d)$,
\begin{equation}\label{Conv_n_m}
\lim_{m\rightarrow\infty}v_{k,n,m}(\mu) \ = \ v_{k,n}(\mu), \qquad\qquad \lim_{n\rightarrow\infty}v_{k,n}(\mu) \ = \ (u\star\gamma_k)(\mu).
\end{equation}
As a matter of fact, regarding the first limit in \eqref{Conv_n_m} we begin noting that
\begin{align*}
h_{k,n,m}(x_1,\ldots,x_n) \ = \ \int_{\R^{dn}}(u\star\gamma_k)\bigg(\frac{1}{n}\sum_{i=1}^n\delta_{x_i-y_i}\bigg)\rho_m(y_1)\cdots\rho_m(y_n)\,dy_1\cdots dy_n \\
\overset{m\rightarrow\infty}{\longrightarrow} \ (u\star\gamma_k)\bigg(\frac{1}{n}\sum_{i=1}^n\delta_{x_i}\bigg),
\end{align*}
for every $x_1,\ldots,x_n\in\R^d$. Then, by the boundedness of $u\star\gamma_k$, we find
\[
\E\big[h_{k,n,m}(\xi_1,\ldots,\xi_n)\big] \ \overset{m\rightarrow\infty}{\longrightarrow} \ \E\bigg[(u\star\gamma_k)\bigg(\frac{1}{n}\sum_{i=1}^n\delta_{\xi_i}\bigg)\bigg].
\]
Similarly, the second limit in \eqref{Conv_n_m}, that is $v_{k,n}(\mu) - u_k(\mu)$ $=$ $\E\big[u\big(\frac{1}{n}\sum_{i=1}^n\delta_{\xi_i}\big)\big] - u(\mu)$ $\overset{n\rightarrow\infty}{\longrightarrow}$ $0$, follows from the continuity of $u$ and the Glivenko-Cantelli convergence in the Wasserstein distance (see \cite[Section 5.1.2]{CD18_I}), namely from $\lim_{n\rightarrow\infty}\Wc_2\big(\frac{1}{n}\sum_{i=1}^n\delta_{\xi_i},\mu\big)$ $=$ $0$, $\P$-a.s.

\vspace{2mm}

\noindent\emph{Step III. The case with $u$ uniformly continuous.} Let us now suppose that $u$ is uniformly continuous. Recall from Step I that $u\star\gamma_k$ is also uniformly continuous with the same modulus $w$. Let $\mu,\nu\in\mathscr P_2(\R^d)$ and consider $n$ independent and identically distributed $\R^d\times\R^d$-valued random variables $(\xi_1,\eta_1),\ldots,(\xi_n,\eta_n)$, with $\xi_i$ (resp. $\eta_i$) having distribution $\mu$ (resp. $\nu$). Then (recalling \eqref{UC_u_rho_k})
\begin{align*}
&|v_{k,n,m}(\mu) - v_{k,n,m}(\nu)| \\
&\leq \ \int_{\R^{dn}}\E\bigg[\bigg|(u\star\gamma_k)\bigg(\frac{1}{n}\sum_{i=1}^n\delta_{\xi_i-y_i}\bigg) - (u\star\gamma_k)\bigg(\frac{1}{n}\sum_{i=1}^n\delta_{\eta_i-y_i}\bigg)\bigg|\bigg]\rho_m(y_1)\cdots\rho_m(y_n)\,dy_1\cdots dy_n \\
&\leq \ \int_{\R^{dn}}\E\bigg[w\bigg(\Wc_2\bigg(\frac{1}{n}\sum_{i=1}^n\delta_{\xi_i-y_i},\frac{1}{n}\sum_{i=1}^n\delta_{\eta_i-y_i}\bigg)\bigg)\bigg]\rho_m(y_1)\cdots\rho_m(y_n)\,dy_1\cdots dy_n
\end{align*}
Now, by \eqref{EstimateW2_bis} and Jensen's inequality (recall that $w$ is concave), we find
\[
\E\bigg[w\bigg(\Wc_2\bigg(\frac{1}{n}\sum_{i=1}^n\delta_{\xi_i-y_i},\frac{1}{n}\sum_{i=1}^n\delta_{\eta_i-y_i}\bigg)\bigg)\bigg] \ \leq \  w\Bigg(\sqrt{\frac{1}{n}\sum_{i=1}^n\|\xi_i - \eta_i\|_{L^2}^2}\Bigg).
\]
Recalling that $(\xi_1,\eta_1),\ldots,(\xi_n,\eta_n)$ are identically distributed, we conclude that $|v_{k,n,m}(\mu) - v_{k,n,m}(\nu)|$ $\leq$ $w(\|\xi_1-\eta_1\|_{L^2})$. From the arbitrariness of $\xi_1,\eta_1$, we find $|v_{k,n,m}(\mu) - v_{k,n,m}(\nu)|$ $\leq$ $w(\|\xi-\eta\|_{L^2})$, $\forall\,\xi,\eta\in L^2$, with $\P_\xi=\mu$ and $\P_\eta=\nu$. By \eqref{W_2_xi_eta} there exists a sequence $\{(\xi_k,\eta_k)\}_k\in L^2(\Omega,\Fc,\P;\R^d)\times L^2(\Omega,\Fc,\P;\R^d)$, with $\P_{\xi_k}=\mu$ and $\P_{\eta_k}=\nu$, such that $\Wc_2(\mu,\nu)=\lim_{k\rightarrow\infty}\|\xi_k-\eta_k\|_{L^2}$. From the continuity of $w$ we conclude that $|v_{k,n,m}(\mu) - v_{k,n,m}(\nu)|$ $\leq$ $w\big(\Wc_2(\mu,\nu)\big)$.

\vspace{2mm}

\noindent\emph{Step IV. Uniform convergence on compact sets and definition of $u_k$.} In this final step we suppose that $u$ is continuous, but not necessarily uniformly continuous, and let $\mathscr M$ be a compact subset of $\mathscr P_2(\R^d)$. Notice that the restriction of $u$ to $\mathscr M$ is uniformly continuous. Then, by the same argument as in Step III we deduce that the double sequence $\{v_{k,n,m}\}_{n,m}$ is equicontinuous on $\mathscr M$. As a consequence, recalling also the convergences \eqref{Conv_n_m}, we deduce from \cite[Lemma D.1]{CossoRusso} that, for every fixed $k\in\N$, there exists a subsequence $\{v_{k,n,m_n}\}_n$ which converges pointwise to $u\star\gamma_k$ on $\mathscr P_2(\R^d)$. Now, notice that the following convergences hold
\[
\lim_{n\rightarrow\infty}v_{k,n,m_n}(\mu) \ = \ (u\star\gamma_k)(\mu), \qquad\qquad \lim_{k\rightarrow\infty}(u\star\gamma_k)(\mu) \ = \ u(\mu), \qquad \forall\,\mu\in\mathscr P_2(\R^d).
\]
Then, we can apply again \cite[Lemma D.1]{CossoRusso} and deduce the existence of a subsequence $\{u_k\}_k:=\{v_{k,n_k,m_{n_k}}\}_k$.
Since $\{u_k\}_k$ is equicontinuous (and clearly also equibounded) on every compact subset $\mathscr M$ of $\mathscr P_2(\R^d)$, by the Ascoli-Arzel\`a theorem it follows that \eqref{Conv_u_n_u} holds.
\end{proof}

\begin{Remark}
Let $u\colon\mathscr P_2(\R^d)\rightarrow\R$ be Lipschitz continuous:
\[
|u(\mu) - u(\nu)| \ \leq \ L\Wc_2(\mu,\nu), \qquad \forall\,\mu,\nu\in\mathscr P_2(\R^d),
\]
for some constant $L\geq0$. Then, it follows from Theorem \ref{T:Approx} that there exists a sequence $\{u_k\}_k$ in $C^\infty(\mathscr P_2(\R^d))$ such that
\[
|u_k(\mu) - u_k(\nu)| \ \leq \ L\Wc_2(\mu,\nu), \qquad \forall\,\mu,\nu\in\mathscr P_2(\R^d),\,\forall\,k\in\N
\]
and
\[
\lim_{k\rightarrow\infty} \sup_{\mu\in\mathscr M} |u_k(\mu) - u(\mu)| \ = \ 0,
\]
for any compact set $\mathscr M\subset\mathscr P_2(\R^d)$.
\end{Remark}

\noindent We end this section proving another main result of the paper, concerning the approximation of first and second-order derivatives.

\begin{Theorem}\label{T:Approx_Deriv}
Let $u\colon\mathscr P_2(\R^d)\rightarrow\R$ be a continuous function. Then, there exists a sequence $\{u_k\}_k\subset C^\infty(\mathscr P_2(\R^d))$ such that \eqref{Conv_u_n_u} holds, moreover 
\begin{enumerate}[\upshape1)]
\item if $u\in C^1(\mathscr P_2(\R^d))$ then
\begin{equation}\label{Conv_der}
\lim_{k\rightarrow\infty} \sup_{(\mu,x)\in\mathscr M\times K} |\partial_\mu u_k(\mu)(x) - \partial_\mu u(\mu)(x)| \ = \ 0, \qquad\qquad 
\end{equation}
for any compact sets $\mathscr M\subset\mathscr P_2(\R^d)$ and $K\subset\R^d$;
\item if in addition $u\in C^2(\mathscr P_2(\R^d))$ then
\begin{align}\label{Conv_der_2nd}
\lim_{k\rightarrow\infty} \sup_{(\mu,x,y)\in\mathscr M\times K \times K} \Big(&\big|\partial_x\partial_\mu u_k(\mu)(x) - \partial_x\partial_\mu u(\mu)(x)\big| \\
&+ \big|\partial_\mu^2 u_k(\mu)(x,y) - \partial_\mu^2 u(\mu)(x,y)\big|\Big) \ = \ 0, \notag
\end{align}
for any compact sets $\mathscr M\subset\mathscr P_2(\R^d)$ and $K\subset\R^d$.
\end{enumerate}
\end{Theorem}
\begin{proof}
We proceed as in Step I of the proof of Theorem \ref{T:Approx} and consider, for every $k\in\N$, $u\star\gamma_k$, which is bounded (together with its derivatives), see \cite[Lemma 5.94]{CD18_I}. We split the rest of the proof into two steps.

\vspace{2mm}

\noindent\emph{Step I. Proof of item 1).} For every $k,n\in\N$, let $v_{k,n}$ be the function given by \eqref{v_n}. Notice that
\begin{equation*}
		\partial_\mu v_{k,n}(\mu)(x) = \E\bigg[\partial_\mu (u\star\gamma_k)\bigg(\frac{1}{n}\sum_{j=2}^n\delta_{\xi_i} + \frac{1}{n}\delta_x\bigg)(x)\bigg].
	\end{equation*}
Now, let us consider the sequence $\{v_{k,n,m}\}_{m\in\N}$ defined in the proof of Theorem \ref{T:Approx}, that is
	\begin{equation*}
		v_{k,n,m}(\mu) \ = \ \E\big[h_{k,n,m}(\xi_1,\ldots,\xi_n)\big],
	\end{equation*}
	where $h_{k,n,m}$ is of class $C^\infty(\R^{dn})$. For every $j=1,\ldots,n-1$, let $v_{k,n,m}^j\colon\R^{dj}\times\mathscr P_2(\R^d)\rightarrow\R$ be given by
\[
v_{k,n,m}^j(x_1,\ldots,x_j,\mu) \ = \ \E\big[h_{k,n,m}(x_1,\ldots,x_j,\xi_{j+1},\ldots,\xi_n)\big], \qquad \forall\,(x_1,\ldots,x_j,\mu)\in\R^{dj}\times\mathscr P_2(\R^d).
\]
Let also $v_{k,n,m}^n\colon\R^{dn}\rightarrow\R$ be given by $v_{k,n,m}^n\equiv h_{k,n,m}$. Notice that
\begin{align*}
v_{k,n,m}(\mu) \ &= \ \int_{\R^d} v_{k,n,m}^1(x,\mu)\,\mu(dx), \\
v_{k,n,m}^j(x_1,\ldots,x_j,\mu) \ &= \ \int_{\R^d} v_{k,n,m}^{j+1}(x_1,\ldots,x_j,x,\mu)\,\mu(dx), \qquad j=1,\ldots,n-1, \\
v_{k,n,m}^{n-1}(x_1,\ldots,x_{n-1},\mu) \ &= \ \int_{\R^d} v_{k,n,m}^n(x_1,\ldots,x_{n-1},x)\,\mu(dx).
\end{align*}
Hence, to obtain the $L$-derivative of $v_{k,n,m}$ we can apply iteratively \cite[formula (5.37)]{CD18_I} together with \cite[Proposition 5.35]{CD18_I}, from which we obtain
	\begin{align*}
		&\partial_\mu v_{k,n,m}(\mu)(x) \\
&= \ \frac{1}{n} \sum_{k=1}^n\E\left[\int_{\R^{dn}}\partial_{\mu}(u\star\gamma_k)\bigg(\frac{1}{n}\sum_{i\neq k}^n\delta_{\xi_i-y_i}+\frac{1}{n}\delta_{x-y_k}\bigg)(x-y_k)\,\rho_m(y_1)\dots\rho_m(y_n)dy_1\dots dy_n\right].
	\end{align*}
By symmetry,
	\begin{align*}
&\partial_\mu v_{k,n,m}(\mu)(x) \\
&= \ \E\left[\int_{\R^{dn}}\partial_{\mu}(u\star\gamma_k)\bigg(\frac{1}{n}\sum_{i= 2}^n\delta_{\xi_i-y_i}+\frac{1}{n}\delta_{x-y_1}\bigg)(x-y_1)\,\rho_m(y_1)\dots\rho_m(y_n)dy_1\dots dy_n\right].
	\end{align*}
Then, for every $\mu\in\mathscr P_2(\R^d)$ and $x\in\R^d$, it holds that
	\begin{equation}\label{Conv_n_m_der}
		\lim_{m\rightarrow\infty}\partial_\mu v_{k,n,m}(\mu)(x) \ = \ \partial_\mu v_{k,n}(\mu)(x), \qquad\qquad \lim_{n\rightarrow\infty}\partial_\mu v_{k,n}(\mu)(x) \ = \ \partial_\mu (u\star\gamma_k)(\mu)(x).
	\end{equation}
The above convergences can be proved proceeding as in Step II of the proof of Theorem \ref{T:Approx}, exploiting the fact that $\partial_\mu(u\star\gamma_k)$ is bounded, and also the following result:
	\begin{equation*}
\lim_{n\rightarrow\infty} \Wc_2\left(\frac{1}{n}\sum_{i=2}^n\delta_{\xi_i} + \frac{1}{n}\delta_x, \mu \right) = 0.
	\end{equation*}
	Indeed, it holds
	\begin{align}\label{EstimateW2_derivative}
		\Wc_2\left(\frac{1}{n}\sum_{i=2}^n\delta_{\xi_i} + \frac{1}{n}\delta_x, \mu \right) \ \leq \ \Wc_2\left(\frac{1}{n}\sum_{i=2}^n\delta_{\xi_i} + \frac{1}{n}\delta_x, \frac{1}{n}\sum_{i=1}^n\delta_{\xi_i} \right) + \Wc_2\left(\frac{1}{n}\sum_{i=1}^n\delta_{\xi_i}, \mu \right).
	\end{align}
	Looking at the right hand side of \eqref{EstimateW2_derivative}, we have that $\left\{\Wc_2\left(\frac{1}{n}\sum_{j=1}^n\delta_{\xi_i}, \mu \right)\right\}_n$ converges almost surely to zero as $n\to\infty$ (see \cite[Section 5.1.2]{CD18_I}). By \eqref{EstimateW2_bis}, it follows that
	\begin{align*}
		\Wc_2\left(\frac{1}{n}\sum_{j=2}^n\delta_{\xi_i} + \frac{1}{n}\delta_x, \frac{1}{n}\sum_{j=1}^n\delta_{\xi_i} \right)\leq \frac{1}{n}\|\xi_1 - x\|_{L^2},
	\end{align*}
	which also tends to zero as $n\to\infty$. Finally, by the pointwise convergences \eqref{Conv_n_m_der} we can proceed as in Step IV of the proof of Theorem \ref{T:Approx} to find a sequence $\{u_k\}_k$ such that \eqref{Conv_der} holds.

\vspace{2mm}

\noindent\emph{Step II. Proof of item 2).} Regarding $\partial^2_\mu v_{k,n,m}$, proceeding as in Step I for the function $\mu\mapsto\partial_\mu v_{k,n,m}(\mu)(x)$, with $x\in\R^d$ fixed, we obtain
	\begin{align*}
		&\partial^2_\mu v_{k,n,m}(\mu)(x,z) \\
&= \E\bigg[\int_{\R^{dn}}\!\!\!\!\!\partial_{\mu}(u\star\gamma_k)\bigg(\!\frac{1}{n}\sum_{i= 3}^n\!\delta_{\xi_i-y_i}\!\!+\!\frac{1}{n}\delta_{x-y_1}\!+\!\frac{1}{n}\delta_{z-y_2}\bigg)\!(x-y_1,z-y_2)\rho_m(y_1)\cdots\rho_m(y_n)dy_1\cdots dy_n\bigg].
	\end{align*}
On the other hand, concerning the derivative $\partial_x\partial_\mu v_{k,n,m}$, for every fixed $y_1,\ldots,y_n\in\R^d$, let us define the map $g_{\xi_1,\ldots,\xi_n,y_1,\ldots,y_n}\colon\R^d\times\R^d\rightarrow\R$ as follows: 
	\begin{equation*}
		g_{\xi_1,\ldots,\xi_n,y_1,\ldots,y_n}(x,z) \ = \ \partial_{\mu}(u\star\gamma_k)\bigg(\frac{1}{n}\sum_{i=2}^n\delta_{\xi_i-y_i}+\frac{1}{n}\delta_{x-y_1}\bigg)(z-y_1), \qquad \forall\,(x,z)\in\R^d\times\R^d.
	\end{equation*}
The partial derivative of $g_{\xi_1,\ldots,\xi_n,y_1,\ldots,y_n}$ with respect to $x$ can be computed thanks to \cite[Proposition 5.35]{CD18_I}:
	\begin{equation}\label{eq: der_g}
		\partial_x g_{\xi_1,\ldots,\xi_n,y_1,\ldots,y_n}(x,z) \ = \ \frac{1}{n}\partial^2_\mu (u\star\gamma_k) \bigg(\frac{1}{n}\sum_{i=2}^n\delta_{\xi_i-y_i}+\frac{1}{n}\delta_{x-y_1}\bigg)(z-y_1,x-y_1).
	\end{equation}
	Thus, noting that
\[
\partial_\mu v_{k,n,m}(\mu)(x) \ = \ \E\left[\int_{\R^{dn}}g_{\xi_1,\ldots,\xi_n,y_1,\ldots,y_n}(x,x)\,\rho_m(y_1)\dots\rho_m(y_n)dy_1\dots dy_n\right],
\]
by the usual chain rule combined with \eqref{eq: der_g}, we obtain
\begin{align*}
&\partial_x\partial_\mu v_{k,n,m}(\mu)(x) \\
&= \frac{1}{n}\E\bigg[\int_{\R^{dn}}\!\!\!\!\partial^2_\mu (u\star\gamma_k) \bigg(\frac{1}{n}\sum_{i=2}^n\delta_{\xi_i-y_i}+\frac{1}{n}\delta_{x-y_1}\bigg)(x-y_1,x-y_1)\,\rho_m(y_1)\cdots\rho_m(y_n)dy_1\cdots dy_n\bigg] \\
&+ \E\bigg[\int_{\R^{dn}}\!\!\!\!\partial_x\partial_\mu (u\star\gamma_k) \bigg(\frac{1}{n}\sum_{i= 2}^n\delta_{\xi_i-y_i}+\frac{1}{n}\delta_{x-y_1}\bigg)(x-y_1)\,\rho_m(y_1)\cdots\rho_m(y_n)dy_1\cdots dy_n\bigg].
\end{align*}
As in Step I, we can conclude that, for every $\mu\in\mathscr P_2(\R^d)$ and $x,z\in\R^d$, 
	\begin{align*}\label{Conv_n_m_der_2nd}
		\lim_{m\rightarrow\infty}\partial^2_\mu v_{k,n,m}(\mu)(x,z) \ &= \ \partial_\mu v_{k,n}(\mu)(x,z), &\qquad \lim_{n\rightarrow\infty}\partial^2_\mu v_{k,n}(\mu)(x,z) \ &= \ \partial^2_\mu (u\star\gamma_k)(\mu)(x,z),	\\
		\lim_{m\rightarrow\infty}\partial_x\partial_\mu v_{k,n,m}(\mu)(x) \ &= \ \partial_x\partial_\mu v_{k,n}(\mu)(x), &\qquad \lim_{n\rightarrow\infty}\partial_x\partial_\mu v_{k,n}(\mu)(x) \ &= \ \partial_x\partial_\mu (u\star\gamma_k)(\mu)(x).
	\end{align*}
Finally, using the pointwise convergences above and proceeding as in Step IV of the proof of Theorem \ref{T:Approx} we can find a sequence $\{u_k\}_k$ such that \eqref{Conv_der_2nd} holds.
\end{proof}
\small

\bibliographystyle{plain}
\bibliography{references}

\end{document}